\documentclass{amsart}

\usepackage[latin1]{inputenc}
\usepackage{enumerate}
\usepackage{a4wide}
\usepackage{mathrsfs}
\usepackage{amsmath}
\usepackage{amssymb}
\usepackage{amsthm}
\usepackage{stmaryrd}
\usepackage{mathtools}

\usepackage{tikz}
\usetikzlibrary{matrix}
\usetikzlibrary{arrows}
\usepackage{pdfsync}

\newtheorem{prop}{Proposition}

\newtheorem{thm}{Theorem}
\newtheorem{cor}{Corollary}

\theoremstyle{definition}

\newtheorem{rem}{Remark}

\usepackage{xspace}
\newcommand{\Q}{\ensuremath{\mathbb{Q}}\xspace}
\newcommand{\X}{\ensuremath{\mathbb{X}}\xspace}

\newcommand{\T}{\ensuremath{\mathbb{T}}\xspace}
\newcommand{\Z}{\ensuremath{\mathbb{Z}}\xspace}
\newcommand{\C}{\ensuremath{\mathbb{C}}\xspace}
\newcommand{\PP}{\ensuremath{\mathbb{P}}\xspace}

\newcommand{\OO}{\ensuremath{\mathcal{O}}\xspace}

\newcommand{\Tn}{\ensuremath{T_\infty}\xspace}

\DeclareMathOperator{\Hom}{Hom}
\DeclareMathOperator{\orb}{orb}
\DeclareMathOperator{\cone}{cone}

\DeclareMathOperator{\Cl}{Cl}

\DeclareMathOperator{\Eff}{Eff}
\DeclareMathOperator{\Mov}{Mov}

\DeclareMathOperator{\Sample}{SAmple}

\begin{document}

\title[The Cox ring of the space of complete rank two collineations]%
{The Cox ring of the space of \\ complete rank two collineations}

\author[J.~Hausen]{J\"urgen Hausen}
\address{Universit\"at T\"ubingen,
Fachbereich Mathematik,
Auf der Morgenstelle 10, 
72076 T\"ubingen, 
Germany}
\email{juergen.hausen@uni-tuebingen.de}

\author[M.~Liebend\"orfer]{Michael Liebend\"orfer}
\address{Universit\"at Kassel,
Fachbereich 10 Mathematik und Naturwissenschaften,
Institut f\"ur Mathematik,
34109 Kassel, 
Germany}
\email{liebend@mathematik.uni-kassel.de}


\maketitle

\section{Introduction}

The space $X(u,c,d)$ of complete collineations 
compactifies the space of isomorphisms between
$u$-dimensional vector subspaces of two given
complex vector spaces $V$ and $W$ of
finite dimension $c$ and $d$ respectively;
it showed up in classical algebraic geometry, 
see~\cite{La1,La2} for some background.
In~\cite{Thaddeus}, one finds the following
precise definition: 
let $S_V$ denote the tautological bundle over 
the Grassmannian $G(u,V)$, consider the rational 
map
$$
\PP\Hom(S_V,W) 
\ \dashrightarrow \ 
\prod_{i=0}^u \PP \Hom(\bigwedge^iS_V,\bigwedge^iW),
\qquad 
f \ \mapsto \ \bigwedge^i f
$$
and define $X(u,c,d)$ to be the closure of the image.
Various descriptions of this space have been given
by Vainsencher~\cite{Vain} and Thaddeus~\cite{Thaddeus}.
Among other things it was shown in~\cite{Thaddeus} 
that $X(u,c,d)$ can be realized as the Chow/Hilbert 
quotient and also as the GIT-limit of a torus action 
on a Grassmannian; this establishes a nice 
analogy to results of Kapranov~\cite{Kap} on the 
space $\overline{M}_{0,n}$ of point configurations 
on the line.

In this note, we determine the Cox ring of the 
space of complete rank 2 collineations.
The Cox ring $\mathcal{R}(X)$ of a complete 
normal variety $X$ with free finitely 
generated divisor class group $\Cl(X)$ is
defined as follows:
take any subgroup $K \subseteq {\rm WDiv}(X)$ 
of the group of Weil divisors projecting 
isomorphically onto $\Cl(X)$ and set
\begin{eqnarray*}
\mathcal{R}(X)
& := & 
\bigoplus_{D \in K} \Gamma(X,\mathcal{O}_X(D)).
\end{eqnarray*}
The interesting cases for the space $X(2,c,d)$ of complete rank 2 
collineations are $c,d>2$; the remaining ones are simple, 
see~Remark~\ref{rem:c2ord2}.
Here comes our result.

\begin{thm}
\label{thmone}
Let $c,d > 2$.
The Cox ring $\mathcal{R}(X(2,c,d))$ of the space  
of complete rank 2 collineations is the factor algebra
$\C[T_{ij},T_\infty; \; 1 \le i < j \le c+d]/I$
with the ideal $I$ generated by
\begin{align*}
\Tn T_{ij}T_{kl}-T_{ik}T_{jl}+T_{il}T_{jk} 
&
\quad 
\text{for } 1 \le i < j < k < l \le c+d 
\text{ with } i,j \le c \text{ and } \ c+1 \le k,l,
\\
\phantom{\Tn}T_{ij}T_{kl}-T_{ik}T_{jl}+T_{il}T_{jk} 
&
\quad 
\text{for } 1 \le i < j < k < l \le c+d 
\text{ different from above}. 
\end{align*}
Moreover, the grading of the Cox ring 
$\mathcal{R}(X(2,c,d))$ by the 
divisor class group $\Cl(X(2,c,d)) = \Z^3$
is given by
$$
\deg(T_\infty) 
\ = \
(0,0,1),
\qquad\qquad
\deg(T_{ij})
\ = \
\begin{cases}
(1,1,-1)
& 
\text{if } i,j \le c,
\\
(1,0,0)
& 
\text{if } i \le c, \, c+1 \le j,
\\
(1,-1,0)
& 
\text{if } c+1 \le i,j.
\end{cases}
$$
\end{thm}

This shows in particular that $X(2,c,d)$ has 
a finitely generated Cox ring, i.e.~it is a 
Mori dream space.
Moreover, the explicit knowledge of the Cox ring 
in terms of generators and relations opens an
access to the geometry; for 
example~\cite[Proposition~4.1]{Hausen} gives 
the following.

\begin{cor}
Let $c,d > 2$.
Set $w_1=(1,1,-1),$ $w_2=(1,0,0),$ $w_3=(1,-1,0),$ $w_4 = (0,0,1)$.
Then the cones of effective, movable and semiample 
divisor classes of $X = X(2,c,d)$ in $\Cl(X) = \Z^3$ 
are 
$$ 
\Eff(X) \ = \ \cone(w_1,w_3,w_4),
\qquad\qquad
\Mov(X) 
\ = \ 
\Sample(X)
\ = \ 
\cone(w_1,w_2,w_3).
$$
\end{cor}

\section{Proof of Theorem~\ref{thmone}}

By the main theorem of~\cite{Thaddeus},
the space $X(2,c,d)$ of complete 
rank~$2$ collineations is obtained by blowing 
up a suitable GIT quotient $X_1$ of the 
Grassmannian $G(2, V \oplus W)$ by the 
$\C^*$-action with the weights
one on $V$ and minus one on $W$.
Our idea is to realize this blow up as a 
neat controlled toric ambient modification 
in the sense of~\cite[Def. 5.4, 5.8]{Hausen}
and then verify that we are in the situation 
of~\cite[Thm.~5.9]{Hausen} which finally 
gives the Cox ring.

We work with the affine cone $\overline{X}$
over $G(2, V \oplus W)$.
It lies in the Pl\"ucker coordinate space 
$\overline{Z} := \bigwedge^2 V \oplus W$
and is defined by the Pl\"ucker relations
$$
P_{ijkl} \ := \ T_{ij}T_{kl}-T_{ik}T_{jl}+T_{il}T_{jk},
\qquad 1 \le i<j<k<l \le c+d.
$$
Note that $\overline{X}$ is invariant under the 
action of the 2-torus $\T^2 = \C^* \times \C^*$ 
on the Pl\"ucker space $\overline{Z}$ given by the 
following weight matrix $Q$, where 
$a^+ := \genfrac{(}{)}{0pt}{}{c}{2}$
and $a^0 := cd$ and 
$a^- := \genfrac{(}{)}{0pt}{}{d}{2}$:
\begin{eqnarray*}
Q
& := &
\begin{bmatrix}
1&\ldots&1&1&\ldots&1&\hphantom{-}1&\ldots&\hphantom{-}1
\\
1&\ldots&1&0&\ldots&0&-1&\ldots&-1
\end{bmatrix}
\\[-12pt]
& &
\begin{matrix} 
\;\;\; \underbrace{\phantom{mmmml}}_{a^+}  
& 
\;\underbrace{\phantom{mmmmll}}_{a^0} 
& 
\underbrace{\phantom{mmmmmml}}_{a^-}
&
\;\; \end{matrix}
\end{eqnarray*}
According to~\cite[Prop.~2.9]{BerHaus}, 
the GIT-quotients of the $\T^2$-action correspond to 
the cones of the GIT-fan in the rational character space 
$\X_\Q(\T^2) = \Q^2$.

\begin{prop}
The $\T^2$-actions on $\overline{Z}$ and $\overline{X}$ have
the same GIT-fan $\Lambda$; it has 
the maximal cones
$\lambda_1 := \cone((1,1),(1,0))$ 
and $\lambda_2 := \cone((1,0),(1,-1))$.
\end{prop}

\begin{proof}
In general, the GIT-fan of a $T$-action on 
an affine variety $\overline{Y}$ consists of the 
GIT-chambers $\lambda(w) \subseteq \X_\Q(T)$, 
where $w \in \X_\Q(T)$.
Each such GIT-chamber is defined as 
$$ 
\lambda(w) 
\ := \ 
\bigcap_{\genfrac{}{}{0pt}{}{y \in \overline{Y},}{w \in \omega_y}} \omega_y,
\qquad
\omega_y \ := \ \cone(\deg(f); 
\; f \in \mathcal{O}(\overline{Y}),
\; f(y) \ne 0).
$$
From this we directly see that $\Lambda$ is the 
GIT-fan of the $\T^2$-action on $\overline{Z}$.
To obtain this as well for the $\T^2$-action on $\overline{X}$,
it is sufficient to show, that $\lambda_1, \lambda_2$ 
are of the form $\omega_{x}$ with suitable 
$x \in \overline{X}$.
For $\lambda_1$ take the point $x$ with the Pl\"ucker 
coordinates $x_{1,2}=x_{1,c+d}=1$ and $x_{i,j}=0$ else;
for $\lambda_2$ take $x$ with $x_{1, c+d}=x_{c+d-1, c+d}=1$ 
and $x_{i,j}=0$ else.
\end{proof}

We denote by $p_i \colon \widehat{Z}_i \to Z_i$
and $p_i \colon \widehat{X}_i \to X_i$
the GIT quotients corresponding to the 
chamber $\lambda_i \in \Lambda$.
With any $w_i$ taken from the 
relative interior of~$\lambda_i$, we realize
$\widehat{Z}_i$ and $\widehat{X}_i$ as sets of 
semistable points
$$
\widehat{X}_i 
\ = \ 
\overline{X}^{ss}(w_i), 
\qquad \qquad 
\widehat{Z}_i 
\ = \ 
\overline{Z}^{ss}(w_i).
$$
Note that $\widehat{Z}_i$, $Z_i$ are toric varieties
and $\widehat{Z}_i \to Z_i$ is a toric morphism. 
Moreover, we have $\widehat{X}_i = \widehat{Z}_i \cap 
\overline{X}$, the quotient map $\widehat{X}_i \to X_i$
is the restriction of $\widehat{Z}_i \to Z_i$ (which justifies 
denoting both of them by $p_i$) and the 
induced maps $X_i \to Z_i$ are closed embeddings.

\begin{rem}
\label{rem:neatquotemb}
Let $c > 2$.
Then the quotients $\widehat{Z}_1 \to Z_1$ and $\widehat{X}_1 \to X_1$
are characteristic spaces, i.e.~relative spectra of Cox sheaves,
see~\cite{ArDeHaLa}.
Moreover, the embedding $X_1 \to Z_1$ is neat in the sense that 
the toric prime divisors of $Z_1$ cut down to prime divisors 
on $X_1$ and the pullback $\Cl(Z_1) \to \Cl(X_1)$ on the level
of divisor class groups is an isomorphism,
see~\cite[Def.~2.5, Prop.~3.14]{Hausen}.
For $d > 2$, one has analogous statements 
on $Z_2$ and $X_2$. 
\end{rem}

\begin{rem}
\label{remmatrix}
The fan $\Sigma_i$ of $Z_i$ is obtained from $\lambda_i$
via linear Gale duality, see for 
example~\cite[Sec~II.2]{ArDeHaLa}. 
The rays of $\Sigma_i$ are generated by the columns 
of any matrix $P$ having the rows of $Q$ as a 
basis of its nullspace; we take 
\begin{align*}
P \ := \
\begin{bmatrix*}
D(a^+)&0&0  \\
0&D(a^0)&0  \\
0&0& D(a^-) \\
A^+&A^0&A^- \\
\end{bmatrix*},
\quad 
\text{where } D(k)\;:= \;
\begin{bmatrix*}[r]
1&-1& &0\\
&\ddots& \ddots &  \\
0&&  1&-1 \\
\end{bmatrix*}\\[-12pt]
\begin{matrix} \underbrace{\phantom{mmm}}_{a^+}\quad \underbrace{\phantom{mmm}}_{a^0}\quad \underbrace{\phantom{mmm}}_{a^-}\qquad \qquad\qquad\qquad\qquad\underbrace{\phantom{mmmmmlmmm}}_{k}\;\; \end{matrix}
\\
\text{and } 
A^+ := (0,\ldots,0,1),
\quad 
A^0 := (-1,0,\ldots,0,-1),
\quad 
A^-:= (1,0,\ldots,0).
\end{align*}
Each maximal cone of $\Sigma_1$ is obtained 
by removing one column from inside the $a^+$-block,
one from outside the $a^+$-block and then taking
the cone over the remaining ones. 
Similarly, for the maximal cones of $\Sigma_2$, 
remove one column from inside the $a^-$-block 
and one from outside. 
\end{rem}

We perform a toric blow up of $Z_1$.
Consider the cone $\sigma_1 \in \Sigma_1$ 
generated by the last $a^-$ colums of $P$
and the barycentric subdivision $\Sigma_\infty$ 
of $\Sigma_1$ at $\sigma_1$; that means that 
we insert the ray $\varrho_\infty$ through 
$(0,\ldots,0,1)$.
Then $\Sigma_\infty \to \Sigma_1$ defines 
a toric blow up $Z_\infty \to Z_1$ centered at
the toric orbit closure $\overline{\orb(\sigma_1)}$ 
corresponding to $\sigma_1$.
Let $X_{\infty} \subseteq Z_\infty$ be the 
proper transform of $X_1 \subseteq Z_1$ under 
this map.

\begin{prop}
The variety $X_{\infty}$ is the space of rank-2 
complete collineations.
\end{prop}

\begin{proof}
The closed subset $\widehat{Z}_1 \setminus \widehat{Z}_2$
of $\widehat{Z}_1$ is the toric orbit closure 
corresponding to the cone $\widehat{\sigma}_1$ 
in $\Q^{a^++a^0+a^-}$ generated by the last 
$a^-$ coordinate axes. 
Note that we have $P(\widehat{\sigma}_1)=\sigma_1$.
Consequently, we obtain 
$$
\overline{\orb(\sigma_1)}
\ = \
p_1(\overline{\orb(\widehat{\sigma}_1)})
\ = \
p_1(\widehat{Z}_1\setminus \widehat{Z}_2).
$$
Hence $X_{\infty}$ is the blow up of $X_1$ at
$X_1 \cap p_1(\widehat{Z}_1 \setminus \widehat{Z}_2)$. 
From~\cite[Sec.~3.3]{Thaddeus} we know, that 
the space of complete collineations is obtained 
from $X_{1}$ by blowing up the subvariety 
$p_{1}(\widehat{X}_{1}\setminus \widehat{X}_{2})$. 
Since $p_1 \colon \widehat{Z}_1 \to Z_1$ is a good 
quotient and $\overline{X}$, 
$\widehat{Z}_1 \setminus \widehat{Z}_2$ are closed 
invariant subsets, we obtain
$$
p_{1}(\widehat{X}_{1}\setminus \widehat{X}_{2})
\ = \
p_1(\overline{X}\cap(\widehat{Z}_1 \setminus \widehat{Z}_2))
\ = \
X_1 \cap p_1(\widehat{Z}_1 \setminus \widehat{Z}_2).
$$
\end{proof}

\begin{prop}
The varieties $X_1,X_\infty$ and the center of 
$\pi_1 \colon X_\infty \to X_1$ are smooth.
\end{prop}

\begin{proof}
This is a result of Vainsencher, see~\cite[Theorem~1]{Vain}. 
We give a simple alternative proof.
By Remark~\ref{remmatrix}, the set of semistable 
points $\widehat{Z}_1$ is covered 
by open affine $p_1$-saturated sets as follows
$$
\widehat{Z}_1 
\ = \
\bigcup_{i<j \leq c,\; k \leq c < l} \overline{Z}_{T_{ij}T_{kl}} 
\cup 
\bigcup_{i<j \leq c <m < n} \overline{Z}_{T_{ij}T_{mn}}.
$$
By the nature of the Pl\"ucker relations,
$\widehat{X}_1$ is smooth and it is contained 
in the first union.
There the torus $\T^2$ acts freely and thus 
the quotient $X_1$ inherits smoothness. 
The inverse image 
$p_1^{-1}(C) = \overline{\orb(\widehat{\sigma}_1)} \cap \overline{X}$ 
of the center 
$C = \overline{\orb(\sigma_1)} \cap X_1$
of $\pi_1$ is explitly given by cutting down
the Pl\"ucker relations.
Calculating the Jacobian of the resulting
equations shows that $p_1^{-1}(C)$ is smooth.
It follows that $C = p_1^{-1}(C) / \T^2$
is smooth.
\end{proof}

\begin{rem}
\label{rem:c2ord2}
For $c > 2$ we obtain $X(2,c,2) \cong X_1$
and 
for $d > 2$ we obtain $X(2,2,d) \cong X_2$.
In both cases Remark~\ref{rem:neatquotemb}
tells us that the Cox ring is $\C[\overline{Z}]$
modulo the Pl\"ucker relations $P_{ijkl}$ endowed
with the $\Z^2$-grading given by the matrix $Q$.
Finally, the Cox ring of $X(2,2,2) = \PP^3$ is 
$\C[T_0,T_1,T_2,T_3]$ with the standard $\Z$-grading.
\end{rem}

From now on we assume $c,d > 2$.
Consider the toric characteristic space 
$p_\infty \colon \widehat{Z}_\infty \to Z_\infty$
obtained via Cox's construction, see~\cite{Co} 
and~\cite[Sec.~II.1.3]{ArDeHaLa}.
Then $\widehat{Z}_\infty$ is an open toric subvariety 
of $\overline{Z}_\infty = \overline{Z} \times \C$.
We denote the additional coordinate by $x_\infty$;
it corresponds to the new ray $\varrho_\infty$.
Moreover, $p_\infty \colon \widehat{Z}_\infty \to Z_\infty$ 
is a good quotient for the $\T^3$-action on 
$\overline{Z}_\infty$ given by the weight matrix
\begin{eqnarray*}
\label{eqgewichte}
Q_\infty
& := &
\begin{bmatrix}
\hphantom{-}1&\ldots&\hphantom{-}1&\hphantom{-}1&\ldots&\hphantom{-}1&\hphantom{-}1&\ldots&\hphantom{-}1&0
\\
\hphantom{-}1&\ldots&\hphantom{-}1&\hphantom{-}0&\ldots&\hphantom{-}0&-1&\ldots&-1&0
\\
-1&\ldots&-1&\hphantom{-}0&\ldots&\hphantom{-}0&\hphantom{-}0&\ldots&\hphantom{-}0&1
\end{bmatrix}
\\[-12pt]
& &
\begin{matrix} 
\;\;\; 
\underbrace{\phantom{mmmmmml}}_{a^+}  
& 
\;\underbrace{\phantom{mmmmmml}}_{a^0} 
& 
\underbrace{\phantom{mmmmmml}}_{a^-}
& 
\;\; 
\end{matrix}
\end{eqnarray*}
Now set $\widehat{X}_\infty := p_\infty^{-1}(X_\infty)$
and write $\overline{X}_\infty$ for the closure 
of $\widehat{X}_\infty$ in $\overline{Z}_\infty$.
Then we arrive at the following commutative diagram:
\begin{center}
\label{diagkomplett}
\begin{tikzpicture}[
        back line/.style={densely dotted},
        normal line/.style={-stealth},
        dash line/.style={dashed},
        cross line/.style={normal line,
           preaction={draw=white, -, 
           line width=6pt}}
    ]
    \matrix (m) [matrix of math nodes, 
         row sep=1.7em, column sep=1.7em,
         text height=1ex, 
         text depth=0.25ex]{
         &\overline{Z}_{\infty}& &\overline{Z}\\
         \overline{X}_{\infty} & &\overline{X}\\
         &Z_{\infty}&&Z_{1}\\
         X_{\infty}   &       & X_{1}   \\
    };
    \path[normal line]
        (m-1-2) edge node[above left]{$\overline{\pi}_{1}$}(m-1-4)
                edge [dash line] node[above left]{$p_{\infty}$} (m-3-2)
        (m-1-4) edge [dash line] node[above right]{$p_{1}$}(m-3-4)
        (m-3-2) edge node[above left]{$\pi_{1}$}(m-3-4)
        (m-2-1) edge [dash line](m-4-1)
                edge (m-2-3)
        (m-2-3) edge [dash line] (m-4-3)
        (m-4-1) edge (m-4-3)
               ;
     \path[right hook->] 
        (m-2-1) edge (m-1-2)
        (m-2-3) edge (m-1-4)
        (m-4-1) edge (m-3-2)
        (m-4-3) edge (m-3-4)
        ;
\end{tikzpicture}
\end{center}
The lifting
$\overline{\pi}_1 \colon \overline{Z}_\infty \to \overline{Z}$ 
is the good quotient for the $\C^*$-action 
on $\overline{Z}_\infty$ having weight $1$ 
at the $x_{i,j}$ with $i,j \leq c$, 
weight $-1$ at $x_\infty$ and weight $0$ else,
see~\cite[Lemma~5.3]{Hausen}.
Its comorphism is given by
\begin{eqnarray*}
\overline{\pi}_1^*(T_{ij})
& = & 
\begin{cases}
T_{\infty}T_{ij} 
&
\text{if } i,j \leq c,
\\ 
\phantom{T_{\infty}}T_{ij} 
&
\text{else}.
\end{cases}
\end{eqnarray*} 
Pulling back the Pl\"ucker relations 
from $\overline{Z}$ to $\overline{Z}_\infty$ 
via $\overline{\pi}_1$ and canceling 
by $\Tn$ whenever possible 
yields the following relations:
\begin{align*}
\Tn T_{ij}T_{kl}-T_{ik}T_{jl}+T_{il}T_{jk} 
&
\quad \text{for } 1 \le i < j < k < l \le c+d 
\text{ with } i,j \le c \text{ and } \ c+1 \le k,l,
\\
\phantom{\Tn}T_{ij}T_{kl}-T_{ik}T_{jl}+T_{il}T_{jk} 
&
\quad \text{for } 1 \le i < j < k < l \le c+d \text{ different from above}. 
\end{align*}
Let $I \subseteq \C[T_{ij},T_\infty\; 1 \le i < j \le c+d]$ 
denote the ideal generated by these polynomials.
The following is essential for showing that 
$R := \C[T_{ij},T_\infty\; 1 \le i < j \le c+d]/I$ 
is the Cox ring of $X_\infty$.

\begin{prop}
\label{prop:Tinftyprime}
The ideal $J := I + \langle T_\infty \rangle 
\subseteq \C[T_{ij},T_\infty; \; 1 \le i < j \le c+d]$
is prime and its zero set $V(J) \subseteq \overline{Z}_\infty$ 
is irreducible of dimension $2(c+d)-3$.
\end{prop}

\begin{proof}
We first show that the ideal  
$J = I + \langle T_\infty \rangle$
is prime.
Obviously, it is generated by the variable 
$T_\infty$ and the polynomials
\begin{eqnarray*}
\label{eqbinome}
g_{ijkl} 
& 
:= 
&
\phantom{T_{ij}T_{kl}}-T_{ik}T_{jl}+T_{il}T_{jk} 
\quad
\text{for } 1 \le i < j < k < l \le c+d 
\text{ with } i,j \le c \text{ and } \ c+1 \le k,l,
\\
h_{ijkl} 
& 
:=
& 
T_{ij}T_{kl}-T_{ik}T_{jl}+T_{il}T_{jk} 
\quad
\text{for } 1 \le i < j < k < l \le c+d \text{ different from above}. 
\end{eqnarray*}
Our task is to show that the polynomials $g_{ijkl}$
and $h_{ijkl}$
generate a prime ideal $\mathfrak{a}$ in the 
polynomial ring $A := \C[T_{ij}; \  1 \le i < j \le c+d]$.
Consider the polynomial ring
\begin{eqnarray*}
B 
& := & 
\C[S_{\alpha\beta}, \, 
   S_\kappa, \, 
   S_{\gamma\delta}; \
   1 \le \alpha < \beta \le c, \
   1 \le \kappa \le c+d, \
   c+1 \le \gamma < \delta \le c+d]
\end{eqnarray*}
and the Segre-type map
$$ 
\sigma \colon A \ \to \ B,
\qquad
T_{ij}
\ \mapsto \
\begin{cases}
S_{ij} & \text{if } 1 \le i < j \le c,
\\
S_iS_j & \text{if } 1 \le i \le c, \ c+1 \le j \le c+d,
\\
S_{ij} & \text{if } \ c+1 \le i < j \le c+d.
\end{cases}
$$
The image $\sigma(A) \subseteq B$ 
equals the degree zero subalgebra $B_0 \subseteq B$ 
arising from the $\Z$-grading on $B$ given by
$\deg(S_{\alpha\beta}) := 0 =: \deg(S_{\gamma\delta})$
always and 
$$ 
\deg(S_\kappa) \ := \ 1, \text{ for } 1 \le \kappa \le c,
\qquad
\deg(S_\kappa) \ := \ -1, \text{ for } c+1 \le \kappa \le c+d.
$$
Moreover, the binomials $g_{ijkl} \in \mathfrak{a}$ 
defined above generate $\ker(\sigma) \subseteq A$. 
For the images of the polynomials $h_{ijkl}$, we obtain
\begin{eqnarray*} 
\sigma(h_{ijkl})
& = &
\begin{cases}
S_{ij}S_{kl}-S_{ik}S_{jl}+S_{il}S_{jk}
&
\text{if }
i,j,k,l \le c,
\\
S_l(S_{ij}S_k-S_{ik}S_j+S_iS_{jk})
&
\text{if }
i,j,k \le c, \ c+1 \le l,
\\
S_i(S_jS_{kl}-S_kS_{jl}+S_lS_{jk})
&
\text{if }
i \le c, \ c+1 \le j,k,l,
\\
S_{ij}S_{kl}-S_{ik}S_{jl}+S_{il}S_{jk}
&
\text{if }
c+1 \le i,j,k,l.
\end{cases}
\end{eqnarray*}
Let $\mathfrak{b}_0 \subseteq B_0$ denote the ideal 
generated by the images $\sigma(h_{ijkl})$.
Then we have $\mathfrak{b}_0 = \sigma(\mathfrak{a})$.
Moreover, consider the ideal 
$\mathfrak{b} \subseteq B$ generated by the polynomials
\begin{align*}
S_{ij}S_{kl}-S_{ik}S_{jl}+S_{il}S_{jk}
&
\quad \text{for }
i,j,k,l \le c, \hfill
\\
S_{ij}S_k-S_{ik}S_j+S_iS_{jk}
&
\quad \text{for }
i,j,k \le c,
\\
S_jS_{kl}-S_kS_{jl}+S_lS_{jk}
&
\quad \text{for }
c+1 \le j,k,l,
\\
S_{ij}S_{kl}-S_{ik}S_{jl}+S_{il}S_{jk}
&
\quad \text{for }
c+1 \le i,j,k,l.
\end{align*}
Note that we have $A/\mathfrak{a} \cong B_0/\mathfrak{b}_0$ 
and $\mathfrak{b}_0 = \mathfrak{b} \cap B_0$.
Thus, it suffices to show that the ideal
$\mathfrak{b} \subseteq B$ is prime. 
For this, define subalgebras
\begin{eqnarray*}
B' 
& := &
\C[S_{\alpha\beta},S_{\kappa};   \ 
    1 \le \alpha < \beta \le c, \ 
    1 \le \kappa \le c],
\\
B'' 
& := &
\C[S_{\gamma\delta},S_{\kappa};   \ 
    c+1 \le \gamma < \delta \le c+d, \ 
    1 \le \kappa \le c].
\end{eqnarray*}
Then we have $B = B' \otimes_\C B''$.
The first ``block'' of generators 
$\sigma(h_{ijkl})$ of $\mathfrak{b}$ lives 
in $B'$; more explicitly, we mean the polynomials
$$
S_{ij}S_{kl}-S_{ik}S_{jl}+S_{il}S_{jk},
\quad
S_{ij}S_k-S_{ik}S_j+S_iS_{jk},
\qquad
\text{where }
i,j,k,l \le c.
$$
Observe that these are in fact Pl\"ucker relations:
setting $e := c+1$ and renaming $S_m$ as $S_{me}$ 
for $1 \le m \le c$, they turn into 
$$
S_{ij}S_{kl}-S_{ik}S_{jl}+S_{il}S_{jk},
\qquad
\text{where }
1 \le i < j < k < l \le e.
$$
In particular, the first block of the $\sigma(h_{ijkl})$ 
generates a prime ideal $\mathfrak{b}' \subseteq B'$. 
Similarly, the second block generates 
a prime ideal $\mathfrak{b}'' \subseteq B''$. 
We have
\begin{eqnarray*}
B/\mathfrak{b}
& \cong & 
B' / \mathfrak{b}' \ \otimes_\C \ B'' / \mathfrak{b}''.
\end{eqnarray*}
Since $\C$ is algebraically closed, we can conclude 
that $B/\mathfrak{b}$ is integral and thus 
$\mathfrak{b} \subseteq B$ is prime.
We showed that $J$ is prime.
 
In order to compute the dimension of $V(J)$,
recall that $B_0 \subseteq B$ is the invariant
algebra of the $\C^*$-action given by the 
$\Z$-grading of $B$. Using this we see
$$ 
\dim(V(J))
\ = \ 
\dim(A/\mathfrak{a})
\ = \ 
\dim(B_0/\mathfrak{b}_0)
\ = \ 
\dim(B/\mathfrak{b})-1
\ = \ 
\dim(B'/\mathfrak{b}') + \dim(B''/\mathfrak{b}'') - 1.
$$
Since the ideals $\mathfrak{b}'$ and $\mathfrak{b}''$
define the cones over the Grassmannians $G(2,c+1)$ 
and $G(2,d+1)$, we obtain at the end
$$ 
\dim(V(J))
\ = \ 
(2(c+1-2) + 1) +  (2(d+1-2) + 1) - 1
\ = \ 
2(c+d)-3.
$$
\end{proof}

\begin{prop}
\label{propiso}
The zero set $V(I) \subseteq \overline{Z}_\infty$ 
is irreducible and equals $\overline{X}_\infty$.
\end{prop}

\begin{proof}
Set for the moment $U := V(I)_{\Tn}$ and
observe that we have an isomorphism
$U \rightarrow \overline{X} \times \C^*$, 
$x \mapsto (\overline{\pi}_1(x), x_{\infty})$.
Consequently, $U$ is irreducible and 
of dimension $2(c+d)-2$. 
Denoting by $\overline{U}$  the closure 
of $U$ in $\overline{Z}_\infty$, we have 
$$ 
\overline{X}_\infty 
\ = \ 
\overline{U},
\qquad
V(I) 
\ = \ 
U \cup V(J) 
\ = \ 
\overline{U} \cup V(J).
$$
Clearly, $\overline{U} \cap V(T_\infty)$ 
is contained in $V(J) = V(I) \cap V(T_\infty)$.
As it contains the origin,
$\overline{U} \cap V(T_\infty)$ is nonempty
and thus of dimension $2(c+d)-3$.
Proposition~\ref{prop:Tinftyprime} yields
$\overline{U} \cap V(T_\infty) = V(J)$.
This shows $V(I) = \overline{U} = \overline{X}_\infty$.
\end{proof}

The last step before proving Theorem $\ref{thmone}$ 
is to ensure, that $\pi_{1}\colon Z_{\infty} \to Z_{1}$
is a neat controlled ambient modification for 
$X_{\infty}\subseteq Z_{\infty}$ and $X_{1} \subseteq Z_{1}$
in the sense of~\cite[Def.~5.4, 5.8]{Hausen}.
{\em Neat\/} means that $X_\infty$ intersects the toric exceptional
divisor $E \subseteq Z_\infty$ in its big torus 
orbit and the pullback of $E$ to $X_\infty$ is a prime 
divisor. {\em Controlled\/} means that the minus 
cell $V(T_\infty) \cap \overline{X}_\infty$ of the $\C^*$-action
on $\overline{X}_\infty$ is irreducible.

\begin{prop}
\label{neatcontr}
The map $\pi_{1}\colon Z_{\infty} \to Z_{1}$ is a 
neat controlled ambient modification for 
$X_{\infty}\subseteq Z_{\infty}$ and $X_{1} \subseteq Z_{1}$.
\end{prop}

\begin{proof}
Set $n := c+d$. Consider the sets
$\widehat{W} := (\widehat{Z}_{\infty})_{T_{12}T_{c,c+1}T_{n-1,n}}$
and $W := p_{\infty}(\widehat{W})$.
Then the monomial 
$\widehat{f} := T_{\infty}T_{12}T_{c,c+1}^{-2}T_{n-1,n}$
gives a local equation for $V(T_\infty)$ on $\widehat{W}$ and, 
by Proposition~\ref{prop:Tinftyprime}, 
it defines a prime element 
on $\widehat{W} \cap \overline{X}_\infty$.
Since $\widehat{f}$ is $\T^3$-invariant,
it descends to a local equation for $E$ on $W$,
and, as it is prime, pulls back to a local equation
for $E \cap X_\infty$ on $W \cap X_\infty$.
To see that $X_\infty$ intersects the big torus orbit 
in the exceptional divisor $E$, observe that 
$p_\infty(x)$ lies in the intersection for $x_{\infty}=0$, 
$x_{ij}=1$ for $i\leq c<j$ and $x_{ij}=j-i$ else. 
This shows that $p_\infty$ is neat.
Propositions~\ref{prop:Tinftyprime} and~\ref{propiso}
guarantee that $p_\infty$ is controlled.
\end{proof}

\begin{proof}[Proof of Theorem~\ref{thmone}]
According to Proposition \ref{neatcontr}, the 
map $\pi_{1}\colon Z_{\infty} \to Z_{1}$ is a 
neat controlled ambient modification for 
$X_{\infty}\subseteq Z_{\infty}$ and $X_{1} \subseteq Z_{1}$.
In this situation, \cite[Thm.~5.9.]{Hausen}
tells us that~$X_{\infty}$ has the normalization of 
$\OO(\overline{X}_{\infty})$ as its Cox ring.
Thus, according to Proposition~\ref{propiso},
we are left with showing that $I$ is a radical 
ideal defining a normal ring.
Let us see that $I$ is a radical ideal.
We have $\sqrt{I} \subseteq J = I + \Tn$, 
because both ideals are radical and 
the reverse inclusion holds for their zero 
sets. Moreover, $\Tn \notin \sqrt{I}$ gives 
$$
\sqrt{I}
\ = \ 
(I+\Tn)\cap \sqrt{I}
\ = \ 
I+\Tn \sqrt{I}.
$$
By induction, we conclude
$\sqrt{I} = I+\Tn^k \sqrt{I}$
for every $k>0$.
Since $\sqrt{I}$ is finitely generated 
and contained in the prime ideal $I_{\Tn}$,
we have $\Tn^k \sqrt{I} \subseteq I$ 
for some $k$. 
This implies $\sqrt{I}=I$.
We show that 
$R := \C[T_{ij},T_\infty\; 1 \le i < j \le c+d]/I$ 
is normal.
The localization $R_{T_\infty}$ is, up to 
a transformation of variables, 
defined by Pl\"ucker relations and thus factorial.
Since $T_\infty \in R$ is prime, 
we obtain that $R$ is factorial as well and 
hence normal.
\end{proof}



\begin{thebibliography}{}
\bibitem{ArDeHaLa}
I.~Arzhantsev, U.~Derenthal, J.~Hausen, A.~Laface:
Cox rings. Preprint, arXiv:1003.4229; see also the 
authors' webpages.
%
\bibitem{BerHaus}
F.~Berchtold and J.~Hausen: 
GIT equivalence beyond the ample cone.
Michigan Math. J. {\bf 54}, 483--515 (2006).
%
\bibitem{Co}
D.A.~Cox:
The homogeneous coordinate ring of a toric variety. 
J. Algebraic Geom. {\bf 4}, 17--50 (1995). 
%
\bibitem{Hausen}
J.~Hausen: 
Cox rings and combinatorics II. 
Moscow Math. J. {\bf 8}, 711--757 (2008).
%
\bibitem{Ca}
A.-M.~Castravet:
The Cox ring of $\overline M_{0,6}$.  
Trans. Amer. Math. Soc.  {\bf 361},  
3851--3878  (2009).
%
\bibitem{Kap}
M.M.~Kapranov:
Chow quotients of Grassmannians. I. 
I.M.~Gel'fand Seminar,  29--110, 
Adv. Soviet Math., {\bf 16}, Part 2, Amer. Math. Soc., 
Providence, RI, 1993.
%
\bibitem{La1}
D.~Laksov: 
Notes on the evolution of complete correlations. 
Enumerative and Classical Algebraic Geometry Nice, 1981, 
Prog. in Math., Birkh\"auser, {\bf 24}, 107--132 (1982).
%
\bibitem{La2}
D.~Laksov: 
Completed quadrics and linear Maps. 
Algebraic geometry: Bowdoin 1985, 
AMS Proc. of Symp. Pure Math., {\bf 46-2}, 371--381 (1987). 
%
\bibitem{Thaddeus}
M.~Thaddeus: Complete collineations revisited. 
Math. Ann. {\bf 315}, 469--495 (1995).
%
\bibitem{Vain}
I.~Vainsencher: 
Complete collineations and blowing up determinantal 
ideals. 
Math. Ann. {\bf 267}, 417--432 (1984).
%
\end{thebibliography}
\end{document}